\documentclass[11pt, a4paper]{amsart}
\usepackage{amsmath}
\usepackage{amssymb}
\usepackage{color}

\newtheorem{prop}{Proposition}[section]
\newtheorem{teo}{Theorem}[section]
\newtheorem{lema}{Lemma}[section]
\newtheorem{coro}{Corollary}[section]

\theoremstyle{definition}

\def\ep{\varepsilon}

\def\a{\mathfrak q}
\def\R{\mathbb R}

\def\K{{\mathcal K}}

\begin{document}
\title[Nonlocal diffusion
on the half line]{Asymptotic behavior for a nonlocal diffusion
equation on the half line}

\author[Cort\'{a}zar, Elgueta, Quir\'{o}s \and Wolanski]{C. Cort\'{a}zar, M. Elgueta, F. Quir\'{o}s \and N. Wolanski}

\address{Carmen Cort\'{a}zar\hfill\break\indent
Departamento  de Matem\'{a}tica, Pontificia Universidad Cat\'{o}lica
de Chile \hfill\break\indent Santiago, Chile.} \email{{\tt
ccortaza@mat.puc.cl} }

\address{Manuel Elgueta\hfill\break\indent
Departamento  de Matem\'{a}tica, Pontificia Universidad Cat\'{o}lica
de Chile \hfill\break\indent Santiago, Chile.} \email{{\tt
melgueta@mat.puc.cl} }

\address{Fernando Quir\'{o}s\hfill\break\indent
Departamento  de Matem\'{a}ticas, Universidad Aut\'{o}noma de Madrid
\hfill\break\indent 28049-Madrid, Spain.} \email{{\tt
fernando.quiros@uam.es} }

\address{Noemi Wolanski \hfill\break\indent
Departamento  de Matem\'{a}tica, FCEyN,  UBA,
\hfill\break \indent and
IMAS, CONICET, \hfill\break\indent Ciudad Universitaria, Pab. I,\hfill\break\indent
(1428) Buenos Aires, Argentina.} \email{{\tt wolanski@dm.uba.ar} }

\thanks{All authors supported by  FONDECYT grants 7090027 and 1110074. The third author supported by
the Spanish Project MTM2011-24696. The fourth author supported by
the Argentine Council of Research, CONICET under the project PIP625,
Res. 960/12 and UBACYT X117. N. Wolanski is a member of CONICET}



\date{}

\begin{abstract}
We study the large time behavior of solutions to a
non-local diffusion equation,  $u_t=J*u-u$ with $J$ smooth, radially symmetric and compactly supported, posed in $\mathbb{R}_+$ with zero Dirichlet boundary conditions. In sets of the form $x\ge \xi t^{1/2}$, $\xi>0$, the outer region, the asymptotic behavior is given by a multiple of the dipole solution for the local heat equation, and the solution is $O(t^{-1})$.  The proportionality  constant is determined from a conservation law, related to the asymptotic first momentum.
On compact sets, the inner region,  after scaling the solution by a factor  $t^{3/2}$, it converges to a multiple of the unique stationary solution of the problem that behaves as $x$ at infinity. The precise proportionality factor is obtained through a matching procedure with the outer behavior. Since the outer and the inner region do not overlap, the matching is quite involved. It has to be done for the scaled function $t^{3/2}u(x,t)/x$, which takes into account that different scales lead to different decay rates.
\end{abstract}

\maketitle

\date{}

\section{Introduction}
\label{Intro} \setcounter{equation}{0}

The purpose of this paper is to study the large-time behavior of the solution
to the nonlocal
problem
\begin{equation}
\label{eq:main}
\begin{cases}
\displaystyle u_t(x,t)=(J*u)(x,t)-u(x,t),\quad&x\ge 0,\ t>0,\\
u(x,t)=0,&x<0,\ t>0,\\
u(x,0)=u_0(x),&x\in\mathbb{R},
\end{cases}
\end{equation}
with a kernel $J$ that is assumed to be a nonnegative continuous
function with unit integral. We will restrict ourselves to the case
where $J$ is smooth, radially symmetric,  with
a compact support contained in  the ball of radius $d$.
Some of our proofs also require that $J$ does not increase in~$\R_+$.

The initial data $u_0$ are assumed to be bounded and
identically zero for $x<0$, and to have a finite mass and a finite first momentum, $\int_{\mathbb{R}_+}
u_0(x)(1+x)\,dx<\infty$. In order to prove our main results, they are also required to have a finite second momentum.

It is easy to prove by means of a fixed point argument that there
exists a unique solution $u\in
C\big([0,\infty);L^1\big(\R_+,(1+x)\,dx\big)\big)$ of problem~\eqref{eq:main}; see for instance \cite{CEQW} for an analogous proof.

Evolution problems with this type of diffusion have been widely
considered in the literature, since they can be used to model the
dispersal of a species by taking into account long-range effects,
\cite{BZ}, \cite{CF}, \cite{Fi}. It has also been  proposed as a
model for phase transitions, \cite{BCh1},  \cite{BCh2}, and, quite
recently, for image enhancement, \cite{GO}.

\medskip

\noindent\textsc{The standard (local) heat equation. } To get an idea of which kind of results we may expect, let us take a look at the local counterpart of our equation, the standard heat equation on the half-line
\begin{equation}
\label{eq:local.heat.equation}
\begin{cases}
u_t=u_{xx},&x\in\R_+,\ t>0,\\
u(0,t)=0,&t>0,\\
u(x,0)=u_0(x),&x\in\R_+.
\end{cases}
\end{equation}
By extending the solution from $\R_+$ to $\R$ as an antisymmetric function, one obtains a representation of the solution in terms of the fundamental solution to the heat equation in the whole real line, $\Gamma$, in the form
$$
u(x,t)=\int_0^{\infty}\left(\Gamma(x-y,t)-\Gamma(x+y,t)\right)u_0(y)\,dy.
$$
This representation may be used to determine the large time behavior of the solution. However, if the problem is nonlocal (or nonlinear) this technique will not work. Hence we prefer a different approach which may be applied to other problems, and in particular to \eqref{eq:main}.

The key point is that the antisymmetric extension of the solution to the whole real line is a solution to the Cauchy problem that has vanishing integral for all times. It is well known that the first term of the asymptotic expansion of solutions of the Cauchy problem for the heat equation is given by a multiple of the fundamental solution $\Gamma$ with the same integral. Since in our case this term vanishes, we have to go further in the asymptotic expansion in order to obtain a nontrivial description of the large time behavior. Fortunately, the next term in the expansion is available if the solution has a finite first momentum; see~\cite{DZ}. It is given in terms of the so-called \emph{dipole} solution, $D(x,t)=\Gamma_x(x,t)$, as
\begin{equation}
\label{eq:outer.behavior.local.he}
t\|u(x,t)+C^*D(x,t)\|_{L^\infty(\R_+)}\to0\quad\text{as }t\to\infty.
\end{equation}
The proportionality constant $C^*$ can be determined in terms of the initial data thanks to the conservation of the first momentum,
$$
M_1(t):=\int_0^\infty u(x,t)x\,dx=\int_0^\infty u_0(x)x\,dx=:M_1^*.
$$
Indeed, since $\int_0^\infty D(x,t)x\,dx=-1/2$, then necessarily $C^*=2M_1^*$.

Let us remark that  $D(x,t)=O(t^{-1})$ on  \emph{outer regions} of the form $|x|\ge \xi t^{1/2}$ for every $\xi>0$. However,  $D(x,t)=O(t^{-3/2})$ for $x$ in compact sets. Hence, \eqref{eq:outer.behavior.local.he} only says that $u=o(t^{-1})$  in the \emph{inner region}, close to the origin, and we have to look for a different scaling in order to have a nontrivial asymptotic profile in compact sets.

Having in mind the dipole solution, we expect a behavior $u(x,t)\approx C x t^{-3/2}$ for $x$ bounded. This idea suggests that we
consider a new variable
\begin{equation}
\label{eq:scale.change}
v(x,t)=\frac{t^{3/2}}{x+1} u(x,t),
\end{equation}
which satisfies
$$
v_t=v_{xx}+\frac{2v_x}{x+1}+\frac32\frac{v}{t}.
$$
The equation is posed in the domain (which is not a
cylinder) $0<x<\varepsilon t^{1/2}$, with data $v(0,t)=0$,
$v(\varepsilon t^{1/2},t)\approx \frac{M_1^*}{2\sqrt{\pi}}$, the latter condition being given by the outer behavior, taking $\varepsilon\to0$.

The last term in the equation is asymptotically small. Hence we expect convergence
to the solution of the ODE
$$
v_{xx}+\frac{2v_x}{{x+1}}=0
$$
satisfying the boundary condition $v(+\infty)=\frac{M_1^*}{2\sqrt{\pi}}$. This solution
turns out to be
$$
v_\infty(x)=\frac{M_1^*}{2\sqrt{\pi}}\frac{x}{(x+1)},
$$
which means a behavior $u(x,t)\approx \frac{M_1^*}{2\sqrt{\pi}}x$ in the original
variables. The precise expected result can be expressed as
$$
\frac {t^{3/2}}{x+1}\Big|u(x,t)-\frac{M_1^*\phi(x)}{2\sqrt{\pi}t^{3/2}}\Big|\to0\quad\text{as }t\to\infty \quad\text{uniformly in compact subsets of }\overline\R_+,
$$
where
$\phi(x)=x$ turns out to be the unique stationary solution to the heat equation in $\R_+$ such that $\phi(0)=0$, $(\phi(x)-x)\in L^\infty(\R_+)$.

This formal computation can be made rigorous. In fact, we are able to combine the inner and the outer behavior in a global approximate asymptotic solution which also gives the large time behavior in a continuum of intermediate scales. \begin{teo}\label{thm:heat}
Assume
$u_0\in L^1(\R_+,(1+x+x^2)\,dx)$. Let $u$ be the solution to~\eqref{eq:local.heat.equation}. Then,
\[
\frac {t^{3/2}}{x+1}\Big|u(x,t)+2M_1^*D(x,t)\Big|\to0\quad\text{as }t\to\infty \quad\text{uniformly in }\overline\R_+.
\]
\end{teo}
We have not found a proof of this fact (not even a statement) in the literature, and we include one here for the sake of completeness.

\medskip

\noindent\emph{Remark. } It is enough to perform the proof for nonnegative  initial data. Indeed, if $u^\pm$ are the solutions with initial data $\{u_0\}_\pm$, then, by the linearity of the equation,  $u=u^+-u^-$. Since $M_1^*=\int_0^\infty\{u_0(x)\}_+x\,dx-\int_0^\infty\{u_0(x)\}_-x\,dx$, the result for general data will follow from the results for $u^+$ and $u^-$. Notice however that in the case of initial data with sign changes it may happen that $M_1^*=0$.  In this situation our result is not optimal (solutions decay faster), and we should look for a different scaling.

\medskip

\noindent\textsc{Main results. }
As observed in~\cite{CERW2},
under the usual parabolic scaling our operator `converges' to
$\partial_t-\a\Delta$, where $\a=\frac1{2}\int_\mathbb{R} J(z)|z|^2\,dz$. Hence, after a certain antisymmetrization procedure, we expect  the outer behavior to be given by a multiple  of the dipole solution to the heat equation with diffusivity $\a$, $D_\a(x,t)=D(x,\a t)$.
How do we determine the proportionality constant $C^*$? The key point is that though the first momentum of the solution is not conserved, it converges to a nontrivial constant $M_1^*$. This asymptotic constant coincides with an invariant of the evolution $\int_0^\infty u(x,t)\phi(x)\,dt$, where
$\phi$ is the unique solution to the stationary problem
\begin{equation}
\label{eq:stationary.problem}
J*\phi=\phi\quad\text{for }x\ge 0,\qquad \phi=0\quad\text{for }-d<x<0,\qquad (\phi(x)-x)\in L^\infty(\R_+).
\end{equation}
Therefore, the proportionality constant is necessarily
$$
C^*=2\underbrace{\int_0^\infty u_0(x)\phi(x)\,dx}_{M_1^*}.
$$

Let us remark that, due to the nonlocal character of the diffusion operator, the antisymmetric extension of the solution in $\R_+$ --with respect to no matter what point--  to the whole real line is not a solution of the Cauchy problem. Hence, the convergence argument is not so direct as the one for the local heat equation. We will need to build up sub and supersolutions having the right common asymptotic behavior.

As for the inner behavior, the scaled variable $v$ given by \eqref{eq:scale.change} should converge to a multiple of the solution $\phi$ to \eqref{eq:stationary.problem}, the proportionality constant being obtained by matching with the outer behavior. Both behaviors can be combined in a unique approximating function.
\begin{teo}\label{thm:main} Let $u_0\in L^1(\R_+,(1+x+x^2)\,dx)\cap L^\infty(\R_+)$. Let $J$ be smooth, radially symmetric, nonincreasing in $\R_+$ with support in the ball $(-d,d)$.  Let $u$ be the solution to \eqref{eq:main}. Then,
\begin{equation}
\label{eq:main.result}
\frac {t^{3/2}}{x+1}\Big|u(x,t)+2M_1^*\frac{\phi(x)}{x}D_\a(x,t)\Big|\to0\quad\text{as }t\to\infty \quad\text{uniformly in }\overline\R_+.
\end{equation}
\end{teo}
Observe that Theorem \ref{thm:main} gives a continuum of intermediate decay rates starting with the global rate $O(t^{-1})$ holding in sets $x\ge \xi t^{1/2}$ with $\xi>0$ all the way up to $O(t^{-3/2})$ that holds on compact sets. These scales are explicitly given by the relation $\frac{t^{3/2}}{x}$. For instance,  $\frac{t^{3/2}}{h(t)}u(x,t)\to \frac{M_1^*\xi}{2\a^{3/2}\sqrt\pi}$ as $t\to\infty$ if $\frac x{h(t)}\to \xi$,  $h(t)\to\infty$ and $h(t)t^{-1/2}\to 0$ as $t\to\infty$.

\medskip

\noindent\emph{Remark. } As  in the local case, the result for initial data without a sign restriction will follow from the result for nonnegative initial data. Hence, in what follows we will always assume, without further mention, that $u_0\ge0$.

\medskip

\noindent\textsc{The problem in the whole space. }
When the problem is posed in the whole space the asymptotic behavior is quite different. Indeed, as proved in \cite{CCR},  for any dimension $N$, the solution $u$ to this problem satisfies
\[
t^{N/2}\|u(\cdot,t)-v(\cdot,t)\|_{L^\infty(\R^N)}\to 0\quad\mbox{as }t\to\infty,
\]
where $v$ is the solution to the heat equation with diffusivity $\a=\frac1{2N}\int_{\mathbb{R}^N} J(z)|z|^2\,dz$ and initial condition $v(x,0)=u(x,0)\in L^1(\R^N)\cap L^\infty(\R^N)$. Therefore,
\[
t^{N/2}\|u(x,t)-M\Gamma_\a(x,t)\|_{L^\infty(\R^N)}\to0\quad\mbox{as }t\to\infty,
\]
with $M=\int_{\mathbb{R}^N} u_0$ and $\Gamma_\a(x,t)=\Gamma(x,\a t)$. This is true even for changing sign solutions, as long as $M\ne0$.
In particular, if $N=1$ the global decay rate  is $O(t^{-1/2})$, much slower than for the problem in the half-line, which, due to the loss of mass at the boundary of the domain,  is $O(t^{-1})$.

\medskip

\noindent\textsc{The problem in domains with holes. }  Let us consider now the case where the spatial domain  is the complement in $\mathbb{R}^N$ of a bounded open set $\mathcal{H}$ with
smooth boundary, with zero data in the hole.  When $N\ge3$,
\begin{equation*}\label{eq:result.large.dimensions}
t^{N/2}|u(x,t)-M^*\phi(x)\Gamma_\a(x,t)|\to 0\quad\mbox{as }
t\to\infty \quad\mbox{uniformly in }\R^N,
\end{equation*}
where $\Gamma_\a(x,t)=\Gamma(x,\a t)$ is the fundamental solution of the heat equation with diffusivity $\a$ determined by the kernel $J$, $\phi$ is the unique solution to
\begin{equation*}\label{eq:stationary}
J*\phi=\phi\quad \mbox{in }\mathbb{R}\setminus\mathcal{H},\qquad
\phi=0\quad\mbox{in }\mathcal{H},\qquad \phi(x)\to 1\quad\text{as }|x|\to\infty,
\end{equation*}
and  $M^*=\int_{\mathbb{R}^N} u_0(x)\phi(x)\,dx$; see~\cite{CEQW}. The quantity $M^*$ turns out to be the asymptotic mass which is nontrivial in large dimensions, $N\ge3$. Notice that the rate of decay of solutions is the same one as for the problem in the whole space.

What happens in dimension $N=1$? If the hole is a \emph{large} interval, with a diameter bigger than the radius of the support of the kernel, the domain has two disconnected components which can be treated independently as problems on a half line (after a translation, and also a reflection, if we consider the component connected to $-\infty$). Therefore, we may apply the results of the present paper to describe the large time behavior for solutions of this problem. The situation is quite different from the one in large dimensions. On the one hand, mass decays to zero. Hence the need of using a different scaling for $u$ in the outer limit, which for the one dimensional case is not any more the one that preserves  mass, but the one that preserves the first momentum.
Another important difference with the case of large dimensions is that the rate of decay on compact sets differs from the global one; see  Theorem \ref{thm:main}. This makes the matching between the inner and the outer behavior more involved, requiring the use of the auxiliary variable~\eqref{eq:scale.change}.

If the size of the hole is small, it does not disconnect the real line: the symmetrization techniques used in this paper cannot be applied, and a different approach is needed. On the other hand, taking into account that the rate of decay for the problem in the whole space is different from the one when there is a big hole, a natural question arises: does the global rate of decay in dimension one depend on the size of the hole? This issues will be the subject of a forthcoming paper.

In the critical case $N=2$,  mass is also asymptotically lost. However, the local counterpart, studied in~\cite{Herraiz}, suggests that the outer limit will still be given by a multiple of the fundamental solution, but with logarithmic time corrections. This case will be treated elsewhere.

\medskip

\noindent\textsc{Organization of the paper. } Section 2 is devoted to the study of the stationary problem~\eqref{eq:stationary.problem}. We prove that there exists a unique stationary solution that behaves as $x$ at infinity. As part of the proof of the uniqueness of the solution, we find that there is no bounded nontrivial stationary solution. In Section 3 we find a conservation law and study the different momenta of the solution.
Sections 4 and 5 are devoted to the study of the asymptotic profile. In Section 4 we consider the far-field limit and in Section 5 we consider the near-field limit, first for our non-local problem~\eqref{eq:main} and then for the standard heat equation.

\medskip

\noindent\textit{Notations. } In what follows we will denote
$$
Lu(x,t):=(J*u)(x,t)-u(x,t), \qquad B_r:=(-r,r).
$$

%%%%%%%%%%%%%%%%%%%%%%%%%%%%%%%%%%%%%%%%%%%%%%%%%%%%%%%%%5
\section{The stationary problem}
\setcounter{equation}{0}

In this section we will prove the existence
of a unique solution to  the stationary problem~\eqref{eq:stationary.problem}. The solution $\phi$ will be obtained as the limit, as $n$ tends to
infinity, of solutions $\phi_n$ of problem
\begin{equation}\label{Pn}
\begin{cases}
\displaystyle\phi_n(x)=\int_0^{n+d} J(x-y)\phi_n(y)\,dy,\quad&0\le x\le n,\\
\phi_n(x)=0,\quad&-d<x< 0,\\
\phi_n(x)=x,\quad&x>n.
\end{cases}
\end{equation}
So we start by constructing such a solution $\phi_n$.

\begin{lema}
Problem~\eqref{Pn} has a unique  solution $\phi_n$. It  satisfies $0\le\phi_n(x)\le x+d$.
\end{lema}
\begin{proof}Uniqueness follows immediately from the maximum principle.

To prove existence we consider the operator
\[
T\phi(x)= \int_0^n
J(x-y)\phi(y)\,dy+\int_n^{n+d}J(x-y)y\,dy\quad\mbox{for }0\le x\le n,
\]
and look for  a fixed point $\hat\phi$ for $T$ in the set
\[
\K=\{\phi\in C([0,n]):  0\le \phi(x)\le
x+d\}.
\]
If $\hat\phi$ is a fixed point for the operator $T$ in $\K$, then it is easily checked that
\begin{equation*}\label{phi}
\phi_n(x)=\begin{cases}
\hat\phi(x),&0\le x\le n,\\
0,&-d<x<0,\\
x,&x>n,
\end{cases}
\end{equation*}
is a solution of problem~\eqref{Pn}.

Since the  kernel $J$ is smooth, the operator $T$ is compact in $C([0,n])$. Hence,  to show that $T$ has a fixed point in the convex set $\K$ we only have to prove that $T(\K)\subset\K$.

Let $\phi \in \K$ and $x\in[0,n]$. We have to consider three different cases:
\begin{itemize}

\item[(i)] If $0\le x-d\le x+d\le n$, then
$$
T\phi(x)=\int_{x-d}^{x+d}J(x-y)\phi(y)\,dy\le \int_{x-d}^{x+d}J(x-y)(y+d)\,dy= x+d.
$$

\item[(ii)]  If $x-d\le 0\le x+d\le n$, then
\[
\begin{aligned}
T\phi(x)&= \int_0^{x+d}J(x-y)\phi(y)\,dy\le\int_0^{x+d}J(x-y)(y+d)\,dy
\\
&=x+d-\int_{x-d}^0 J(x-y)(y+d)\,dy \leq x+d.
\end{aligned}
\]

\item[(iii)]  If $x-d<n<x+d<n+d$, then, no matter wether $x-d$ is larger or smaller than $0$,
\[
\begin{aligned}
T\phi(x)&=\int_0^{n} J(x-y)\phi(y)\,dy+ \int_n^{x+d}J(x-y)y\,dy\\
&\le \int_0^{x+d}J(x-y)(y+d)\,dy\le x+d.
\end{aligned}
\]
\end{itemize}
\end{proof}

\begin{prop}
The stationary problem~\eqref{eq:stationary.problem} has  a solution $\phi$  such that
\begin{equation}
\label{eq:estimate.phi}
x\le\phi(x)\le x+d\quad\text{for }x\ge 0.
\end{equation}
\end{prop}

\begin{proof} The function $\psi(x)=x$ satisfies $J*\psi=\psi$ for  $x\ge 0$. On the other hand,
$\phi_n(x)\ge \psi$ if $-d<x< 0$ and $\phi_n(x)=\psi(x)$ if $n<x<n+d$. Therefore,  by comparison,
$$
\phi_n(x)\ge x\qquad \text{for }-d<x<n+d,\ n>0.
$$
We thus have $\phi_n(x)=x\le \phi_{n+1}(x)$ for $n<x<n+d$.
Hence, since both $\phi_n$ and $\phi_{n+1}$ vanish in $(-d,0)$, we conclude, again by comparison, that
$\phi_n(x)\le\phi_{n+1}(x)$ if $x\in(-d,n+d)$. Since, moreover,
$\phi_n(x)  \leq x+d$ for all  $n\in\mathbb{N}$, the monotone limit
$$
\phi(x)=\lim_{n\to\infty}\phi_n(x)
$$
is finite and satisfies~\eqref{eq:estimate.phi}. It is then easily checked that $\phi$ solves \eqref{eq:stationary.problem}.
\end{proof}

\noindent\emph{Remark. }  Notice that as $d$ approaches 0 the stationary solution $\phi$ converges uniformly to the stationary solution of the local problem.

\medskip

Let $\phi_1,\phi_2$ be solutions to \eqref{eq:stationary.problem}.  Then, $\phi=\phi_1-\phi_2$ is  a solution to
\begin{equation}\label{stationary-bounded}
J*\phi=\phi\quad\mbox{for }x\ge 0,\qquad
  \phi=0\quad\mbox{if }-d<x<0,\qquad
   \phi\in L^\infty(\R_+).
\end{equation}
If $J>0$ in $(-d,d)$, we will prove that $\phi\equiv0$, hence uniqueness for problem~\eqref{eq:stationary.problem}. Our proof requires to show that if $\phi$ is close to its supremum $\phi^+$ at a certain point $\bar x$, then  the measure of the points within distance $d$ to $\bar x$ for which $\phi$ is not close to $\phi^+$ is small. This is the content of the next technical lemma.

\begin{lema}
\label{lem:technical}
Let $J>0$ in $B_d$ and  $\gamma = \min_{|z|<\frac34\,d} J(z) $.
Given $\ep>0$ and $\bar x\ge 0$ such that $\phi(\bar x)\ge \phi^+-\ep^4$, we have
$$
\left|\left\{x\ge0: |x-\bar x|<d,\,\phi(x)\le \phi^+-\ep\right\}\right|\le 3\ep/\gamma.
$$
\end{lema}

\begin{proof}
Since
\[
\begin{aligned}
\phi^+-\ep^4&\le\overbrace{\int_{\{\phi(y)\le \phi^+-\ep^2\}}J(\bar x-y)\phi(y)\,dy+\int_{\{\phi(y)> \phi^+-\ep^2\}}J(\bar x-y)\phi(y)\,dy}^{\phi(\bar x)}\\
&\le (\phi^+-\ep^2)\int_{\{\phi(y)\le \phi^+-\ep^2\}}J(\bar x-y)\,dy+\phi^+\int_{\{\phi(y)> \phi^+-\ep^2\}}J(\bar x-y)\,dy\\
&= \phi^+-\ep^2\int_{\{\phi(y)\le \phi^+-\ep^2\}}J(\bar x-y)\,dy\\
&\le \phi^+-\ep^2\gamma\left|\left\{x\ge 0: |x-\bar x|<\frac{3d}4,\ \phi(x)\le \phi^+-\ep^2\right\}\right|,
\end{aligned}
\]
we get that
$|\left\{x\ge 0: |x-\bar x|<\frac{3d}4,\ \phi(x)\le \phi^+-\ep^2\right\}|\le \ep^2/\gamma$. In particular, if $d/4>\ep^2/\gamma$, there exist $z_1\in [\bar x+\frac d2,\bar x+\frac34\,d]$ and $z_2\in[\bar x-\frac34\,d,\bar x-\frac d2]$ such that $\phi(z_i)>\phi^+-\ep^2$, $i=1,2$.

With the above reasoning with $\bar x$ replaced by $z_i$, i=1,2 we find that
\[
\left|\left\{x\ge 0: |x-z_i|<\frac{3d}4,\ \phi(x)\le \phi^+-\ep\right\}\right|\le \ep/\gamma,\quad i=1,2,
\]
and the result follows.
\end{proof}

We may now proceed to the proof of uniqueness.
\begin{prop}
\label{lemma:uniqueness.stationary}
If $J>0$ in $B_d$, the unique solution to~\eqref{stationary-bounded} is $\phi=0$.
\end{prop}

\begin{proof}
Given $x\ge 0$, we define
$$
F(x)=\frac1{C_0}\int_0^d\int_{x-\omega}^{x+\omega}\phi(y)\int_\omega^dJ(z)\,dz\,dy\,d\omega,\qquad C_0=2\int_0^d\omega\int_\omega^dJ(z)\,dz\,d\omega.
$$
It is easy to see (just differentiate) that $F''(x)=\frac1{C_0}L\phi(x)=0$ for $x\ge 0$. Therefore, there exist constants $A, B$, such that $F(x)=Ax+B$,
$x\ge 0$.
Since
$$
\underbrace{\inf_{x\ge 0}\phi(x)}_{\phi^-}\le\phi(x)\le\underbrace{\sup_{x\ge 0}\phi(x)}_{\phi^+},
$$
then $\phi^-\le F(x)\le \phi^+$ if $x\ge 0$. Therefore, $A=0$ and hence  $F(x)=B\in[\phi^-,\phi^+]$.

Let $\bar x\ge 0$ such that $\phi(\bar x)\ge \phi^+-\ep^4$. Using Lemma~\eqref{lem:technical},
\[\begin{aligned}
\phi^+\ge B=F(\bar x)=&\frac1{C_0}\int_0^d\int_{\{|y-\bar x|<\omega\,,\,\phi(y)> \phi^+-\ep\}}\phi(y)\int_\omega^dJ(z)\,dz\,dy\,d\omega\\
&+\frac1{C_0}\int_0^d\int_{\{|y-\bar x|<\omega\,,\,\phi(y)\le \phi^+-\ep\}}\phi(y)\int_\omega^dJ(z)\,dz\,dy\,d\omega\\
\ge&(\phi^+-\ep)\frac1{C_0}\int_0^d\int_{\bar x-\omega}^{\bar x+\omega}\int_\omega^dJ(z)\,dz\,dy\,d\omega\\
&+\frac1{C_0}\int_0^d\int_{\{|y-\bar x|<\omega\,,\,\phi(y)\le \phi^+-\ep\}}\big(\phi(y)-(\phi^+-\ep)\big)\int_\omega^dJ(z)\,dz\,dy\,d\omega\\
\ge&(\phi^+-\ep)-\frac{C_1\phi^+}{C_0}\left|\{|y-\bar x|<d\,,\,\phi(y)\le\phi^+ -\ep\}\right|\\
\ge&\phi^+-\left(\frac{3C_1\phi^+}{C_0\gamma}+1\right)\ep,
\end{aligned}
\]
where $C_1=\int_0^d\int_\omega^d J(z)\,dzd\omega$. Letting $\ep\to0$ we get
$F(x)=B=\phi^+$ in $x\ge 0$.  This is only possible if $\phi(x)=\phi^+$ for $x\ge 0$. But this contradicts the strong maximum principle unless $\phi^+=0$, which finally implies that $\phi\equiv0$.
\end{proof}

%%%%%%%%%%%%%%%%%%%%%%%%%%%%%%%%%%%%%%%%%%%%%%%%%%%%%%%%%
\section{Conservation law, mass decay and study of the first and second momentum}
\label{sec:momenta}
\setcounter{equation}{0}

As explained in the introduction, solutions to~\eqref{eq:main} do not conserve neither mass, nor the first momentum. But there is an important difference between both magnitudes: while mass decays to 0, the first momentum approaches a nontrivial limit.  This limit, which plays a role in the characterization of the large time asymptotics of the solution, can be expressed in terms of the initial data thanks to a certain conservation law. We will prove all these facts in this section. In addition, we will also study the second momentum, since the error term in our asymptotic expansion will be given in terms of this quantity.

\noindent\textsc{Conservation law and mass decay. } We start by deriving a conservation law, which will next be used to prove that mass decays to 0.
\begin{prop} Let $u$ be the solution to \eqref{eq:main} and let
$\phi$ be the solution to the stationary problem~\eqref{eq:stationary.problem}. Then, for every $t>0$,
\begin{equation*}
\label{cl-big}
M_\phi(t):=\int_0^\infty u(x,t)\phi(x)\,dx=\underbrace{\int_0^\infty u_0(x)\phi(x)\,dx}_{M_1^*}.
\end{equation*}
\end{prop}
\begin{proof}
Since
$u\in C\big([0,\infty):L^1\big(\R_+,(1+x)\,dx\big)\big)$, the estimate~\eqref{eq:estimate.phi} on $\phi$ implies that $M_\phi(t)<\infty$. In addition, using the equation in~\eqref{eq:main}, we get  $\int_0^\infty
|u_t(x,t)|\phi
(x)\,dx<\infty$. Therefore,
we may differentiate under the integral sign to obtain, after applying Tonelli's Theorem,
$$
M'_\phi(t)=\int_0^\infty u_t(x,t)\phi(x)\,dx=\int_0^\infty Lu(x,t)\phi(x)\,dx=
\int_0^\infty u(x,t)L\phi(x)\,dx=0.
$$
\end{proof}

A first estimate on the size of $u$ comes from comparison with the solution $u^c$ of the Cauchy problem having as initial data
$$
u^c_0(x)=\begin{cases}
u_0(x),& x\ge 0,\\
0,&x<0.\end{cases}
$$
This yields   $0\le u(x,t)\le u^c(x,t)\le Ct^{-1/2}$; see, for instance, \cite{CCR,IR1}. This bound allows to obtain an estimate on the mass at time $t$ that shows that it decays to 0 as $t$ grows to infinity.
\begin{prop}
\label{mass-decay-rate} Let $u$ be the solution to \eqref{eq:main}.
Then $M(t)=\int_0^\infty u(x,t)\,dx\to0$ as $t\to\infty$.
\end{prop}
\begin{proof} Let $\delta=\delta(t)>0$, $t>0$,  to be chosen later. Using the estimate~\eqref{eq:estimate.phi} on $\phi$, we get
\[
\begin{aligned}
\int_0^\infty u(x,t)\,dx&=\int_0^{\delta(t)}u(x,t)\,dx+\int_{\delta(t)}^\infty u(x,t)\,dx\\
&\le \int_0^{\delta(t)} C t^{-1/2}\,dx+\frac1{\delta(t)}\int_{\delta(t)}^\infty u(x,t)x\,dx\\
&\le C\delta(t) t^{-1/2}+ \frac1{\delta(t)}\int_0^\infty u_0(x)\phi(x)\,dx.
\end{aligned}
\]
Taking $\delta(t)=t^{1/4}$ we get $M(t)\le Ct^{-1/4}$, hence the result.
\end{proof}

\noindent\textsc{First momentum. }
We can also determine the asymptotic first momentum, which coincides with the conserved quantity.

\begin{prop}\label{Mlimits} Let $u$ be the solution to \eqref{eq:main}.  Then, as $t\to\infty$,
$$
M_1(t)=\int_0^\infty u(x,t)x\,dx\to M_1^*.
$$
\end{prop}
\begin{proof} It follows easily from the estimate~\eqref{eq:estimate.phi} on $\phi$ and the decay of the mass, Proposition~\ref{mass-decay-rate}. Indeed,
\begin{equation}
\label{eq:decay.first.momentum}
\begin{aligned}
|M_1(t)-M_1^*|&\le\int_0^\infty
u(x,t)|x-\phi(x)|\,dx\le d\int_0^\infty
u(x,t)\,dx\to0\quad\mbox{as }t\to\infty.
\end{aligned}
\end{equation}
\end{proof}

\noindent\textsc{Second momentum. } Our next goal is to prove that the second momentum of a solution stays finite for all $t>0$ if the initial second momentum is finite. To this aim we use again that $0\le u \le u^c$.
Hence,
$$
M_2(t):=\int_0^\infty u(x,t)x^2\,dx\le\underbrace{\int_0^\infty u^c(x,t)x^2\,dx}_{M_2^c(t)}.
$$

In order to estimate $M_2^c(t)$ we express $u_c$ in terms of the  fundamental solution $F=F(x,t)$ to
the operator $\partial_t-L$ in the whole space. This fundamental solution can be decomposed as
\begin{equation*}\label{fund-sol}
F(x,t)=e^{-t}\delta(x)+\omega(x,t),
\end{equation*}
where $\delta(x)$ is the Dirac mass at the origin in $\R^N$ and $\omega$ is a smooth function defined via its Fourier transform,
\begin{equation*}
\label{eq:transform.omega}
\hat\omega(\xi,t)=e^{-t}\big(e^{\hat J(\xi)t}-1\big);
\end{equation*}
see~\cite{CCR}. Thus,
$$
u^c(x,t)=e^{-t}u_0(x)+\int_0^\infty\omega(x-y,t)u_0(y)\,dy.
$$
Using the Taylor series of the exponential we get
\begin{equation}
\label{eq:series.omega}
\omega(x,t)=e^{-t}\sum_{n=1}^\infty\frac{t^nJ^{*n}(x)}{n!},\qquad
J^{*n}=\underbrace{J*\dots*J}_{n\text{ times }}.
\end{equation}
This expression, recently derived
in~\cite{BCF}, was used by the authors to prove that
\begin{equation}
\label{eq:estimate.omega}
0\le \omega(x,t)\le C_1 e^{-\frac1 d|x|\log|x|+C_2|x|+|x|\log t}
\end{equation}
for some constants $C_1$ and $C_2$. Estimate \eqref{eq:estimate.omega} implies in particular that all the momenta of $\omega(\cdot,t)$, $t>0$, are finite. This is the key to prove that the second momentum of any solution to~\eqref{eq:main} stays finite for all times if it is initially finite.

\begin{prop}
Let $u$ be the solution to \eqref{eq:main}. Assume $\int_0^\infty u_0(x)x^2\,dx<\infty$. Then, for every $t>0$,
$\int_0^\infty u(x,t)x^2\,dx<\infty$
\end{prop}

\begin{proof}
We have
$$
M_2(t)\le \underbrace{e^{-t}\int_0^\infty u_0(x)x^2 \,dx}_{\mathcal{A}}+\underbrace{\int_0^\infty \int_0^\infty \omega(x-y,t)u_0(y)x^2\,dy\,dx}_{\mathcal{B}}.
$$
By assumption, $\mathcal{A}<\infty$. As for $\mathcal{B}$,  we have
$$
\begin{aligned}
\mathcal{B}=&\int_0^\infty u_0(y)\int_{-y}^\infty\omega(z,t)(y+z)^2\,dz\,dy
\\
\le&\int_0^\infty u_0(y)y^2\,dy\int_{-\infty}^\infty\omega(z,t)\,dz+2\int_0^\infty u_0(y)y\,dy\int_0^\infty\omega(z,t)z\,dz\\
&\ +\int_0^\infty u_0(y)\,dy \int_{-\infty}^\infty\omega(z,t)z^2 \,dz<\infty.
\end{aligned}
$$
\end{proof}

\noindent\emph{Remark. } An analogous argument proves that the first $k$ momenta of the solution stay finite for all time if they were initially finite.

%%%%%%%%%%%%%%%%%%%%%%%%%%%%%%%%%%%%%%%%%%%%%%%%%%%%%%%%%
\section{Far-field limit}
\setcounter{equation}{0}

From now on, in addition to the hypotheses stated in the Introduction,  we assume that $J$ is nonincreasing in $\R_+$.
As  super- and
sub-solutions  we will use  solutions to the problem in the whole space taking as initial data
an antisymmetric extension of the solution at some time $t_0$ with respect to a certain point. Let us start by studying the properties of such functions.

\begin{prop}
\label{prop:antisymmetry.positivity}Let $u$ be the  solution to the nonlocal heat equation $u_t=Lu$ in one spatial dimension with initial data $u_0$.  If $u_0$ is antisymmetric with respect to $\bar x\in \mathbb{R}$ and $u_0(x)\ge0$ for $x\ge \bar x$, then
$u(x,t)$ is antisymmetric in $x$ with respect to $\bar x$ and $u(x,t)\ge 0$ for all $x\ge \bar x$.
\end{prop}

\begin{proof}
Since the convolution of radially symmetric functions which are nonincreasing in  $\R_+$ inherits these two properties (this is a result of Wintner (\cite{W}) that can be easily proved), the regular part $\omega$ of the fundamental solution to the Cauchy problem is also radially symmetric  and nonincreasing in $\R_+$ in the spatial variable; see \eqref{eq:series.omega}.
The antisymmetry of $u$ in $x$ is then an immediate consequence of the antisymmetry of $u_0$, the symmetry of $\omega$ and the fact that $$
u(x,t)=e^{-t}u_0(x)+\int_{\mathbb{R}}\omega(x-y,t)u_0(y)\,dy.
$$
From this expression we also get, using the antisymmetry of $u_0$,
$$
u(x,t)=e^{-t}u_0(x)+\int_{\bar x}^\infty(\omega(x-y,t)-\omega(x+y-2\bar x,t)u_0(y)\,dy.
$$
Hence, if $x\ge \bar x$,
$$
\begin{aligned}
u(x,t)=e^{-t}u_0(x)&+\int_{\bar x}^x(\omega(x-y,t)-\omega(x+y-2\bar x,t)u_0(y)\,dy\\
&+\int_{x}^\infty(\omega(y-x,t)-\omega(x+y-2\bar x,t)u_0(y)\,dy.
\end{aligned}
$$
Since
$$
\begin{aligned}
&\omega(x-y,t)-\omega(x+y-2\bar x,t)\ge0\quad\mbox{if }x\ge y\ge \bar x,\\
&\omega(y-x,t)-\omega(y+x-2\bar x,t)\ge0\quad\mbox{if }y\ge x\ge \bar x,
\end{aligned}
$$
we finally obtain that $u(x,t)\ge0$ for all $x\ge \bar x$.
\end{proof}

We may now proceed to the proof of the outer behavior of  $u$.
\begin{teo}\label{asym1}
Let
$\a=\frac1{2}\int_{\mathbb{R}} J(z) z^2\,dz$ and $u_0\in L^1(\mathbb{R}_+,(1+x+x^2)dx)$. If $u$ is the solution to~\eqref{eq:main}, then
$$
t\|u(\cdot,t)+2M^*_1 D(\cdot,qt)\|_{L^\infty(\R_+)}\to 0 \quad\text{as }t\to\infty.
$$
\end{teo}
\begin{proof} \textsc{Supersolution. }
Let $u_0^+$ be the antisymmetric extension of $u(x,t_0)$ to the whole space with respect to $\bar x=-d$,
$$
u^+_0(x)=\begin{cases}
u(x,t_0)\quad&\mbox{if } x\ge 0,\\
0\quad&\mbox{if }-2d<x<0,\\
-u(-2d-x,t_0)\quad&\mbox{if }x\le -2d,
\end{cases}
$$
and
$u^+(x,t)$ the solution to
\begin{equation}
\label{eq:cauchy.problem}
v_t=Lv\quad\mbox{in }\R\times(t_0,\infty),
\end{equation}
with initial data $u^+(x,t_0)=u_0^+(x)$.
Since, $u^+(x,t)\ge0=u(x,t)$ if $x\in(-d,0)$, see Proposition~\ref{prop:antisymmetry.positivity},
by comparison we have that
\begin{equation*}\label{comparison-above}
u^+(x,t)\ge u(x,t)\quad\mbox{for } x\in\R_+.
\end{equation*}

Since the integral of $u_0^+$ is 0,
\begin{equation}\label{M1-super}
\begin{aligned}
M^+_1(t_0)&=\int_{\mathbb{R}} u^+_0(x)x\,dx=\int_{\mathbb{R}} u^+_0(x)(x+d)\,dx\\
&
= 2 \underbrace{\int_{0}^\infty u(x,t_0)x\,dx}_{M_1(t_0)}+2d\underbrace{\int_{0}^\infty u(x,t_0)\,dx}_{M(t_0)}.
\end{aligned}
\end{equation}
Using the results of Section~\ref{sec:momenta} we then obtain that $M^+_1(t_0)\to 2M_1^*$.

\noindent\textsc{Subsolution. } This time $u_0^-$ is the antisymmetric extension with respect to the origin of $u(x,t_0)$, $x\ge0$,
$$
u^-_0(x)=
\begin{cases}
u(x,t_0)\quad&\mbox{if }x> 0,\\
-u(-x,t_0)\quad&\mbox{if }x<0,
\end{cases}
$$
and $u^-$ the solution to~\eqref{eq:cauchy.problem} with initial data $u^-(x,t_0)=u_0^-(x)$.
Since, $u^-(x,t)\le0=u(x,t)$ if $x\in(-d,0)$, by
comparison we get
\begin{equation*}\label{comparison-below}
u^-(x,t)\le u(x,t)\quad\mbox{for } x\in\R_+.
\end{equation*}

On the other hand, since
$u^-_0$ is anti-symmetric with respect to the origin,
\begin{equation}\label{M1-sub}
M^-_1(t_0)=\int_{\mathbb{R}}u^-_0(x)x\,dx=2\int_0^\infty u(x,t_0)\,x\,dx=2M_1(t_0)\to 2M_1^*\quad\text{as }t_0\to\infty.
\end{equation}

\noindent\textsc{Asymptotics for the barriers. }
Since $\int_{\mathbb{R}} u^\pm_0(x)\,dx=0$ and $u_0^\pm\in L^1(\mathbb{R},(1+x+x^2)dx)$, we have
\begin{equation}
\label{eq:Liviu.Rossi}
(t-t_0)| u^\pm(x,t)+M^\pm_1(t_0)D(x,\a(t-t_0))|\le C\left(\int_{\mathbb{R}}\left| u^\pm _0(x)\right|x^2\,dx\right)(t-t_0)^{-1/2},
\end{equation}
see Theorem 1.1 in \cite{IR1}. Hence,
$$
t\|u^\pm(\cdot,t)+M^\pm_1(t_0)D_\a(\cdot,t)\|_{L^\infty(\R)}\to 0\quad\mbox{as } t\to\infty.
$$
We now use that  $|M_1^\pm(t_0)-2M^*_1|<C\ep$ if $t_0$ is large. Thus, since $t D_\a(x,t)$ is
uniformly bounded, for  $x\in\R_+$ we have
$$
\begin{aligned}
-C\ep<
&\liminf_{t\to\infty}\Big\{t\Big(u^-(x,t)+M^-_1(t_0) D_\a(x,t)\Big)- (M^-_1(t_0)-2M^*_1) t D_\a(x,t)\Big\}\\
\le
&\liminf_{t\to\infty}\Big\{t\Big( u(x,t)+2M^*_1 D_\a(x,t)\Big)\Big\}
\le\limsup_{t\to\infty}\Big\{t\Big( u(x,t)+2M^*_1 D_\a(x,t)\Big)\Big\}\\
\le
&\limsup_{t\to\infty}\Big\{t\Big( u^+(x,t)+M^+_1(t_0) D_\a(x,t)\Big)- (M^+_1(t_0)-2M^*_1) t D_\a(x,t)\Big\}
<C\ep.
\end{aligned}
$$
\end{proof}

As an immediate consecuence of  Theorem~\ref{asym1}  we obtain the size estimate
\begin{equation}
\label{eq:improved.size.u}
\|u(\cdot,t)\|_{L^\infty(\mathbb{R}_+)}=O\left(t^{-1}\right).
\end{equation}
This will now be used to get  improved asymptotic estimates for the mass and the first and second momenta.

\begin{coro}
\label{cor:decay.rates.momenta}
Under the hypotheses of Theorem~\ref{asym1},
\begin{equation}\label{eq:decay.rates.momenta}
M(t)=O\big(t^{-1/2}\big),\quad |M_1(t)-M_1^*|=O\big(t^{-1/2}\big),\quad M_2(t)=O\big(t^{1/2}\big)\quad \text{as }t\to\infty.
\end{equation}
\end{coro}

\begin{proof}

To obtain the decay estimate for the mass, we just repeat the proof of Proposition~\ref{mass-decay-rate}, this time using~\eqref{eq:improved.size.u}  to bound the integral on $(0,\delta(t))$, and taking $\delta(t)=t^{1/2}$. The estimate for $M_1(t)$ then follows from~\eqref{eq:decay.first.momentum}.

Concerning the estimate of the second momentum, since $Lx^2=c$,
\[
\begin{aligned}
M_2'(t)&=\int_0^\infty Lu(x,t) x^2\,dx= \int_{\R}Lu(x,t) x^2\,dx-\int_{-\infty}^0\int_0^\infty J(x-y)u(y,t)x^2\,dy\,dx\\
&= \int_{\mathbb{R}} u(x,t)\,Lx^2\,dx-\int_{-\infty}^0\int_0^\infty J(x-y)u(y,t)x^2\,dy\,dx\\
&\le  c\int_{\mathbb{R}} u(x,t)\,dx\le C t^{-1/2}.
\end{aligned}
\]
Integrating in $(0,t)$, and using that $M_2(0)<\infty$, the result follows.
\end{proof}

These estimates allow us to obtain an estimate for the error in the far-field asymptotics.
\begin{teo} Under the assumptions of Theorem~\ref{asym1},
\begin{equation}
\label{eq:decay.rate.profile}
t\|u(\cdot,t)+2M^*_1 D_\a(\cdot,t)\|_{L^\infty(\mathbb{R}_+)}=O\big(t^{-1/4}\big)
\quad
\text{as }t\to\infty.
\end{equation}
\end{teo}

\begin{proof}
Using~\eqref{M1-super} and~\eqref{M1-sub} and the first two estimates in~\eqref{eq:decay.rates.momenta}, we get
$$
\begin{aligned}
&|M_1^+(t_0)-2M_1^*|\le 2dM(t_0)+2|M_1(t_0)-M_1^*|=O\big(t_0^{-1/2}\big),\\
&|M_1^-(t_0)-2M_1^*|=2|M_1(t_0)-M_1^*|=O\big(t_0^{-1/2}\big).
\end{aligned}
$$
We now use~\eqref{eq:Liviu.Rossi} and the third estimate in~\eqref{eq:decay.rates.momenta} to get that
$$
(t-t_0)| u^\pm(x,t)+M^\pm_1(t_0)D_\a(x,t-t_0)|=O\big(t_0^{1/2}(t-t_0)^{-1/2}\big).
$$
Thus, for $t\ge t_0$,
$$
\begin{aligned}
&-Ct_0^{1/2}(t-t_0)^{-1/2}-\bar Ct_0^{-1/2}\\
&\qquad<
\Big\{(t-t_0)\Big(u^-(x,t)+M^-_1(t_0) D_\a(x,t-t_0)\Big)- (M^-_1(t_0)-2M^*_1) (t-t_0) D_\a(x,t-t_0)\Big\}\\
&\qquad\le
\Big\{(t-t_0)\Big( u(x,t)+2M^*_1 D_\a(x,t-t_0)\Big)\Big\}\\
&\qquad\le
\Big\{(t-t_0)\Big(u^+(x,t)+M^+_1(t_0) D_\a(x,t-t_0)\Big)- (M^+_1(t_0)-2M^*_1) (t-t_0) D_\a(x,t-t_0)\Big\}\\
&\qquad<\tilde Ct_0^{1/2}(t-t_0)^{-1/2}+ \bar Ct_0^{-1/2}.
\end{aligned}
$$
Taking $t_0=t^{1/2}$ and $t\geq 2$,
and observing that
$$
t|D_\a(x,t-t^{1/2})-D_\a(x,t)|=O\big(t^{-1/2}\big),
$$
we get the stated decay.
\end{proof}

This estimate for the error allows to improve the rate of decay in sets of the form $x\ge \mu t^\beta$, $\beta>1/4$; the precise rate will depend on the spatial scale. This improvement will be crucial when performing the matching with the inner behavior in the next section. Hence, we write the result in the form of Theorem~\ref{thm:main}.

\begin{coro}
\label{coro:outer.behavior.wider}
 Under the assumptions of Theorem~\ref{asym1}, for any $\mu>0$ and $\beta>1/4$  there is a constant $C>0$ such that
\begin{equation*}
\label{eq:decay.rate.profile.wider.region}
\frac{t^{3/2}}{x+1}|u(x,t)+2M^*_1\frac{\phi(x)}{x} D_\a(x, t)|\le Ct^{\frac14-\beta} \qquad\text{for all }x\ge \mu t^\beta,\ t>0.
\end{equation*}
\end{coro}

\begin{proof} On sets of the mentioned form we have, on the one hand,
$\frac{t^{1/2}}{x+1}\le \frac{
t^{1/2-\beta}}\mu$ .
On the other hand, since $D_\a(\cdot,t)=O(t^{-1})$,  estimate~\eqref{eq:estimate.phi} on $\phi$ yields
$$
\frac{t^{3/2}}{x+1}\left|D_\a(x,t)\right|\left|\frac{\phi(x)}x-1\right|\le  Ct^{2(\frac14-\beta)}.
$$
The result then follows using~\eqref{eq:decay.rate.profile}.
\end{proof}

\bigskip

%%%%%%%%%%%%%%%%%%%%%%%%%%%%%%%%%%%%%%%%%%%%%%%%%%%%%%%%%
\section{Near-field limit}
\setcounter{equation}{0}

In view of Corollary~\ref{coro:outer.behavior.wider}, what is left to complete the proof of
Theorem~\ref{thm:main} is to show that the limit~\eqref{eq:main.result} is
valid uniformly in sets of the form $0\le x<\mu t^\beta$ for some
$\mu>0$, $\beta>1/4$.
Since $D_\a(x,t)=-\frac{x}{2\a t}\Gamma_\a(x,t)$ and
\begin{equation}
\label{eq:omega.gamma}
t^{1/2}\|
\omega(\cdot,t)-\Gamma_\a(\cdot,t)\|_{L^\infty(\R)}\to0\quad\mbox{as
}t\to\infty,
\end{equation}
see~\cite{IR1}, and $\phi(x)/(x+1)$ is bounded,  this will follow from the next result.

\begin{teo}\label{thm:inner.behavior}
Under the assumptions of Theorem~\ref{thm:main}, for any  $\beta\in\left(\frac14,\frac12\right)$ and $\mu>0$ we have
\begin{equation*}
\label{eq:thm.inner.behavior}
\sup_{0\le x\le\mu t^\beta}\left(\frac{t^{3/2}}{x+1}\left|u(x,t)-\frac{M^*_1\phi(x)\omega(x,t)}{\a t}\right|\right)\to 0
 \qquad\text{as }t\to\infty.
\end{equation*}
\end{teo}
Notice that the same kind of argument, combined with Corollary~\ref{coro:outer.behavior.wider}, shows that for any $\beta>1/4$ and all $\mu>0$ we have
\begin{equation}
\label{eq:inner.omega}
\sup_{x\ge \mu t^\beta}\left(\frac{t^{3/2}}{x+1}\left|u(x,t)-\frac{M^*_1\phi(x)\omega(x,t)}{\a t}\right|\right)\to 0
\quad\text{as }t\to\infty.
\end{equation}
The advantage of this formulation in terms of $\omega$ is that  it is more straightforward to apply the nonlocal operator $L$ to $\omega$ than to $D_\a(x,t)/x$.

In order to prove Theorem~\ref{thm:inner.behavior} we will construct suitable
barriers approaching the asymptotic limit as $t$ goes to infinity.
We choose $\kappa\in(0,1)$,
$\gamma\in(0,1)$, and then define, for any $K_\pm>0$,
\begin{equation*}
\label{eq:def.parabolic.barriers} v_\pm(x,t)=\frac{\phi(x)\omega(x,t)}t\pm
K_\pm t^{-\frac {3+\kappa}{2}}z(x),\qquad z(x)=(x+2d)^{\gamma}.
\end{equation*}
Our barriers will be adequate multiples of $v_\pm$.

\begin{lema}
\label{lem:v.plus.minus}
Let $\kappa,\gamma\in(0,1)$  and  $v_\pm$ as above. For all
$\beta< \frac{1-\kappa}{2(2-\gamma)}$, $\mu>0$ and $K_\pm\ge1$, there is a value $t^*=t^*(\mu,\beta,\a,\kappa,\gamma)$ such that
\begin{equation}
\label{eq:lem.v.plus.minus}
\partial_tv_+-Lv_+\ge 0,\quad \partial_t v_{-}-Lv_-\le0, \qquad 0\le x\le \mu t^\beta,\ t\ge t^*.
\end{equation}
\end{lema}

\begin{proof}
On the one hand,
$$
\partial_t v_+(x,t)=\frac{\phi(x)\partial_t \omega(x,t)}{t}-\frac{\phi(x) \omega(x,t)}{t^2} -\frac {3+\kappa}{2t}K_+t^{-\frac {3+\kappa}{2}}z(x).
$$
On the other hand, using Taylor's expansion and the radial symmetry of
$J$, we obtain
\begin{equation*}\label{potencia}
Lz(x)\le -\frac{\a\gamma(1-\gamma)}2(x+3d)^{\gamma-2}\le-\frac{\a\gamma(1-\gamma)z(x)}{2(x+3d)^2} .
\end{equation*}
Hence,  since $\phi$ is $L$-harmonic, we get
$$
\begin{aligned}
Lv_+(x,t)=& \frac{\phi(x)L\omega(x,t)}t+\frac1t\int_{\mathbb{R}} J(x-y)\big(\phi(y)-\phi(x)\big)\big(\omega(y,t)-\omega(x,t)\big)\,dy\\
&+K_+t^{-\frac{3+\kappa}{2}}L z(x)\qquad \text{for all }x\ge 0.
\end{aligned}
$$
Therefore, since $\omega$ solves
\begin{equation*}\label{eq-W}
\begin{cases}
\partial_t\omega(x,t)-L\omega(x,t)=e^{-t}J(x)\quad&\mbox{in }\R\times(0,\infty),\\
\omega(x,0)=0\quad&\mbox{in }\R,
\end{cases}
\end{equation*}
 and $\phi\ge0$, we get
$$
\begin{aligned}
\partial_t v_+-Lv_+\geq&-\underbrace{\frac1t\int_{\mathbb{R}} J(x-y)\big|\phi(y)-\phi(x)\big|\big|\omega(y,t)-\omega(x,t)\big|\,dy}_{\mathcal{A}}\\
&+K_+\underbrace{t^{-\frac{3+\kappa}{2}}z(x)\left(\frac{\a\gamma(1-\gamma)}{2(x+3d)^2}-\frac {3+\kappa}{2t}\right)}_{\mathcal{B}}-\underbrace{\frac{\phi(x)\omega(x,t)}{t^2}}_{\mathcal{C}}.
\end{aligned}
$$
Thanks to~\eqref{eq:estimate.phi}, we have
$ |\phi(x)-\phi(y)|\le
2d$  if  $|x-y|\le d$.
On the other hand, $|\omega_x(x,t)|\le c t^{-1}$; see~\cite{IR1}. Therefore,
$\mathcal{A}\le C_1t^{-2}$.

As for $\mathcal{B}$, if $0<x<\mu t^\beta$ and $t$ is large,  then
$$
\frac1{(x+3d)^{2}}\ge\frac1{4\mu^2t^{2\beta}}.
$$
Since  $\beta<1/2$, if $t$ is large enough, how large depending only on $\mu$, $\beta$, $\a$, $\kappa$ and $\gamma$,
we have
$$
\mathcal{B}\ge C_2t^{-\frac{3+\kappa}2+(\gamma-2)\beta}
$$
for some constant $C_2>0$.

Finally, using again the estimate~\eqref{eq:estimate.phi}, together with the bound $|\omega(x,t)|\le c t^{-1/2}$ \cite{IR1},  we obtain $\mathcal{C}\le C_3t^{\beta-\frac52}$.

The above estimates for $\mathcal{A}$, $\mathcal{B}$ and $\mathcal{C}$ yield, if $K_+\ge1$,
$$
\partial_t v_+-Lv_+\ge - C_1t^{-2}+C_2t^{-\frac{3+\kappa}2+(\gamma-2)\beta}- C_3t^{\beta-\frac52}>0,
$$
if $\beta<\frac{1-\kappa}{2(2-\gamma)}$ and $t$ is large enough, how large not depending on $K_+$.

An analogous argument leads to the statement concerning $v_-$,
since $$
0\le\frac{ \phi(x)e^{-t}J(x)}t\le Ct^{\beta-1}e^{-t}, \qquad 0\le x\le \mu t^\beta.
$$
\end{proof}

\begin{proof}[Proof of Theorem~\ref{thm:inner.behavior}.]
We take $\kappa\in(0,1)$ small enough and $\gamma\in(0,1)$ close enough to 1 so that $\beta<\frac{1-\kappa}{2(2-\gamma)}$. Then, as a consequence of Lemma~\ref{lem:v.plus.minus}, we know that there is a value $t^*$ such that~\eqref{eq:lem.v.plus.minus} holds.

Let now $\ep>0$. Since $\beta>1/4$, \eqref{eq:inner.omega} implies that there is a time $t_\ep$ such that
$$
u(x,t)-\frac{M_1^*\phi(x)\omega(x,t)}{\a t}\le \ep \frac{x+1}{t^{3/2}}\le \ep 2(\mu+1)t^{\beta-\frac32}\qquad\text{if }\mu t^\beta\le x\le(\mu+1) t^\beta,\ t\ge t_\ep.
$$
Besides, $t^{\frac32-\beta}|D_\a(x,t)|\ge c_{\a,\mu}>0$ if $\mu t^\beta\le x\le (\mu+1)t^\beta$. Therefore, using~\eqref{eq:omega.gamma} once more,
$$
\frac{x\omega(x,t)}{\a t}\ge c_{\a,\mu}t^{\beta-\frac32}\qquad\text{if }\mu t^\beta\le x\le(\mu+1) t^\beta,\ t\text{ large}.
$$
Hence, since $1\le \phi(x)/x$ and $\beta<1/2$, there is a large time $t^+\ge t^*$ such that
$$
u(x,t)-\frac{M_1^*\phi(x)\omega(x,t)}{\a t}\le C\ep \frac{M_1^*\phi(x)\omega(x,t)}{\a t}\qquad\text{if }\mu t^\beta\le x\le(\mu+1) t^\beta,\ t\ge t^+.
$$
We conclude that
for any $K_+\ge0$,
$$
u(x,t)\le (1+C\ep)\frac{M_1^*}\a v_+(x,t)\quad\mbox{if }\mu t^\beta\le x\le(\mu+1) t^\beta,\ t\ge t^+.
$$

On the other hand, we trivially have
$
u(x,t)=0\le (1+C\ep)\frac {M_1^*}\a v_+(x,t)$, if  $x\in(-d,0)$, $t\ge t^+$.
Finally, it is obvious that there exists $K_+\ge1$ such that $u(x,t^+)\le (1+C\ep)\frac {M_1^*}\a v_+(x,t^+)$ if $0\le x\le (\mu+1) (t^+)^{\beta}$.

Putting everything together, and using comparison, we get
$$
u(x,t)\le (1+C\ep)\frac{M_1^*}\a v_+(x,t),\quad0\le x\le  \mu t^\beta,\ t\ge t^+.
$$
Hence, using the decay estimate $0\le \omega(x,t)\le t^{-1/2}$ and the upper estimate on $\phi$,
$$
\begin{array}{rcl}
\displaystyle\frac{t^{3/2}}{x+1}\left(u(x,t)-\frac{M_1^*\phi(x)\omega(x,t)}{\a t}\right)&\le& \displaystyle C\ep \frac{M_1^*}\a\frac{\phi(x)}{x+1}\omega(x,t)t^{1/2}+(1+C\ep)\frac{M_1^*}\a K_+t^{-\frac\kappa2}\frac{z(x)}{1+x}\\
&\le&\displaystyle C\ep +(1+C\ep)C t^{-\frac\kappa2},
\end{array}
$$
if $0\le x\le  \mu t^\beta$, $t\ge t^+$.
Letting $t\to\infty$ and then $\ep\to0$, we conclude that
$$
\limsup\limits_{t\to\infty}\sup_{0\le x\le \mu t^\beta}\left(\frac{t^{3/2}}{x+1}\Big(u(x,t)-\frac{M^*\phi(x)\omega(x,t)}{\a t}\Big)\right)\le 0.
$$

An analogous argument shows that
$$
\liminf\limits_{t\to\infty}\sup_{0\le x\le \mu t^\beta}\left(\frac{t^{3/2}}{x+1}\Big(u(x,t)-\frac{M^*\phi(x)\omega(x,t)}{\a t}\Big)\right)\ge 0.
$$
\end{proof}

We may use the same ideas to obtain the inner
behavior for the standard heat equation. In this case we have a better bound for the error in the outer expansion,  and things are easier.

\begin{proof}[Proof of Theorem~\ref{thm:heat}.]
Let $u$ be the solution to the local problem \eqref{eq:local.heat.equation}. The results from \cite{DZ} yield
\begin{equation}
\label{eq:outer.heat}
|u(x,t)+2M_1^*D(x,t)|\le C t^{-3/2}.
\end{equation}
for some constant $C>0$.
Therefore, for any $\beta>0$, $\mu>0$,
\begin{equation*}
\label{eq:outer.heat.rates}
\frac{t^{3/2}}{x}|u(x,t)+2M_1^*D(x,t)|\le Ct^{-\beta}\quad \text{if }x\ge\mu t^\beta.
\end{equation*}
Unfortunately, this does not give the desired behavior on compact sets, since this would require to take $\beta=0$.
To obtain the inner behavior, we use as barriers suitable multiples of the functions
$$
v_\pm(x,t)=-D(x,t)\pm K_\pm t^{-\frac{3+\kappa}2}(x+1)^\gamma,\quad \kappa,\gamma\in(0,1), K_\pm\ge 1.
$$
An easy computation shows that
$$
\partial_t v_\pm(x,t)-\partial_{xx}v_\pm(x,t)=\pm K_\pm t^{-\frac{3+\kappa}2}(x+1)^\gamma\left(\frac{\gamma(1-\gamma)}{(x+1)^2}-\frac{3+\kappa}{2t}\right).
$$
Thus, $v_+$ and $v_-$ are respectively a sub- and a super-solution of the heat equation in the set
$$
0<x<\underbrace{\left(\frac{\gamma(1-\gamma)}{2(3+\kappa)}\right)^{1/2}}_{\mu_*}t^{1/2},\quad t>t_\mu:=\mu^{-1/\beta}.
$$

Let $\mu<\mu_*$, and $\ep>0$. Since $-tD(\mu t^{1/2},t)=c_\mu>0$, using~\eqref{eq:outer.heat} we get that
there is a time $t_\ep\ge t_\mu$ such that
$$
u(\mu t^{1/2},t)+2M_1^*D(\mu t^{1/2},t)\le Ct^{-3/2}\le-\ep 2M_1^*D(\mu t^{1/2},t)\quad\text{if }t\ge t_\ep.
$$
Thus,
$$
u(\mu t^{1/2},t)\le (1+\ep)2M_1^*v_+(\mu t^{1/2},t)\quad\text{if }t\ge t_\ep.
$$
We also trivially have $u(0,t)=0\le(1+\ep)2M_1^*v_+(0,t)$ for all $t>0$. Finally, if $K_+$ is large enough,
$u(x,t_\ep)\le (1+\ep)2M_1^*v_+(x,t_\ep)$ for $0<x<\mu t_\ep^{1/2}$. Using comparison, we conclude that
$$
u(x,t)\le (1+\ep)2M_1^*v_+(x,t), \quad 0<x<\mu t^{1/2},\ t\ge t_\ep.
$$
An analogous computation yields
$$
u(x,t)\ge (1-\ep)2M_1^*v_-(x,t), \quad 0<x<\mu t^{1/2},\ t\ge t_\ep.
$$
The proof follows letting first $t\to\infty$ and then $\ep\to0$.
\end{proof}

\end{document}